\documentclass[a4paper, 11pt]{article}

\usepackage[english]{babel}
\usepackage[utf8x]{inputenc}
\usepackage{amsmath}
\usepackage{bm}
\usepackage{soul}
\usepackage[a4paper,top=3cm,bottom=2cm,left=3cm,right=3cm,marginparwidth=1.75cm]{geometry}

\usepackage{amsthm}
\usepackage{amsfonts}
\usepackage{graphicx}
\usepackage[colorlinks=true, allcolors=blue]{hyperref}
\usepackage{url}
\usepackage{todonotes}

\newtheorem{corollary}{Corollary}
\newtheorem{definition}{Definition}
\newtheorem{lemma}{Lemma}
\newtheorem{proposition}{Proposition}
\newtheorem{theorem}{Theorem}
\newtheorem{remark}{Remark}
\newtheorem{example}{Example}

\def\NN{\mathbb{N}}
\def\QQ{\mathbb{Q}}

\def\ZZ{\mathbb{Z}}

\title{Simultaneous approximations to $p$--adic numbers and algebraic dependence via multidimensional continued fractions}

\author{Nadir Murru and Lea Terracini \\
Department of Mathematics G. Peano, University of Torino\\
Via Carlo Alberto 10, 10123, Torino, Italy\\
nadir.murru@unito.it, lea.terracini@unito.it}
\date{}

\begin{document}
\maketitle

\begin{abstract}
Unlike the real case, there are not many studies and general techniques for providing simultaneous approximations in the field of $p$--adic numbers $\mathbb Q_p$. Here, we study the use of multidimensional continued fractions (MCFs) in this context. MCFs were introduced in $\mathbb R$ by Jacobi and Perron as a generalization of continued fractions and they have been recently defined also in $\mathbb Q_p$. We focus on the dimension two and study the quality of the simultaneous approximation to two $p$-adic numbers provided by $p$-adic MCFs, where $p$ is an odd prime. Moreover, given algebraically dependent $p$--adic numbers, we see when infinitely many simultaneous approximations satisfy the same algebraic relation.
This also allows to give a condition that ensures the finiteness of the $p$--adic Jacobi--Perron algorithm when it processes some kinds of 
$\mathbb Q$--linearly dependent inputs.
\end{abstract}



\section{Introduction}

Continued fractions give a representation for any real number by means of a sequence of integers, providing along the way rational approximations. In particular, they provide best approximations, i.e., the $n$--th convergent of the continued fraction of a real number is closer to it than any other rational number with a smaller or equal denominator. Multidimensional continued fractions (MCFs) are a generalization of classical continued fractions introduced by Jacobi \cite{Jac} and Perron \cite{Per} in an attempt to answer  a question posed by Hermite about a possible generalization of the Lagrange theorem for continued fractions to other algebraic irrationalities. A MCF is a representation of a $m$-tuple of real numbers $(\alpha_0^{(1)}, \ldots, \alpha_0^{(m)})$  by means of $m$ sequences of integers $((a_n^{(1)})_{n\geq 0}, \ldots, (a_n^{(m)})_{n\geq 0})$ (finite or infinite) obtained by the Jacobi--Perron algorithm:
\begin{equation} \label{eq:JP}
\begin{cases} a_n^{(i)} = [\alpha_n^{(i)}], \quad i = 1, ..., m, \cr
\alpha_{n+1}^{(1)} = \cfrac{1}{\alpha_n^{(m)} - a_n^{(m)}}, \cr
\alpha_{n+1}^{(i)} = \cfrac{\alpha_n^{(i-1)} - a_n^{(i-1)}}{\alpha_n^{(m)} - a_n^{(m)}}, \quad i = 2, ..., m, \end{cases} n = 0, 1, 2, ...
\end{equation}
We shall write
\[ (\alpha_0^{(1)}, \ldots, \alpha_0^{(m)}) = [(a_0^{(1)}, a_1^{(1)}, \ldots), \ldots, (a_0^{(m)}, a_1^{(m)}, \ldots)]. \]
The Jacobi--Perron algorithm has been widely studied concerning its periodicity and approximation properties. For instance, in \cite{Ber}, \cite{Ber2}, \cite{Lev}, \cite{Raj} the authors provided some classes of algebraic irrationalities whose expansion by the Jacobi--Perron algorithm becomes eventually periodic. In \cite{Mur}, a criterion of periodicity, involving linear recurrence sequences, is given. The periodicity of the Jacobi--Perron algorithm is also related to the study of Pisot numbers \cite{Dub1}, \cite{Dub2}. Further studies on MCFs can be found in \cite{Adam}, \cite{Das}, \cite{Mur2}, \cite{Zhu}. 

Continued fractions for $p$-adic numbers were introduced  by several authors \cite{Bro1}, \cite{Rub}, \cite{Schn} and more recently they have been generalized to higher dimensions. In \cite{MT}, the authors studied the fundamental properties of MCFs in $\mathbb Q_p$, focusing on convergence properties and finite expansions, whereas in \cite{MT2} further properties regarding finiteness and periodicity of the $p$--adic Jacobi--Perron algorithm have been proved. 

The study of simultaneous approximations of real numbers is a very important topic in Diophantine approximation; classical and fundamental results can be found in \cite{Adams}, \cite{Cus}, \cite{Dav1}, \cite{Dav2}, \cite{Lag}. Some results can also be found regarding simultaneous approximations in $\mathbb Q_p$, involving a $p$--adic number and its integral powers \cite{Buda}, \cite{Mahler}. Specific results regarding the case of a $p$-adic number and its square are investigated in \cite{Beg} and \cite{Teu}. However, there are no general techniques for providing simultaneous approximations of $p$-adic numbers and for studying the quality of such approximations. 

MCFs have been deeply studied in this context for the real case, since they provide simultaneous rational approximations to real numbers. The quality of these simultaneous approximations has been studied in several works, such as \cite{Bal}, \cite{Che}, \cite{Lag2}, \cite{Pal}, \cite{Sch}, thus, it seems natural to exploit MCFs in $\mathbb Q_p$ for approaching the problem of constructing simultaneous approximations to $p$--adic numbers. In this paper, we give a first study in this direction and we also investigate the relation between simultaneous approximations and algebraic dependence.

The paper is structured as follows. In Section \ref{sec:def}, we introduce the notation and we give some basic definitions and properties. Section \ref{sec:approx} is devoted to the study of the quality of the simultaneous approximations provided by $p$-adic  MCFs. Finally, in Section \ref{sec:alg}, we focus on algebraically dependent pairs of $p$--adic numbers; firstly we find a condition on the quality of approximation under which a sequence of simultaneous rational approximations satisfies the same algebraic relation. Secondly, we apply this result to MCFs and deduce a condition that ensures the finiteness of the $p$--adic Jacobi--Perron algorithm when it processes some kinds of $\mathbb Q$--linearly dependent inputs.

\section{Definitions and useful properties} \label{sec:def}

In the following, we will focus on $p$--adic MCFs of dimension 2, i.e., using the notation of the previous section, we set $m = 2$. Most of the results obtained in this paper can be adapted to  any dimension $m \geq 2$, but in the general case the notation is very annoying and possibly confusing. Hence, we now recall the $p$--adic Jacobi--Perron algorithm for the case $m=2$; for more details see \cite{MT}. From now on, $p$ will be an odd prime number. 

\begin{definition}
The \emph{Browkin $s$-function} $s:\QQ_p\longrightarrow \mathcal{Y} =\ZZ\left [\frac 1 p\right ]\cap \left (-\frac p2,\frac p 2\right)$, is defined by
$$s(\alpha)= \sum_{j=k}^0 x_jp^j,$$
with $\alpha\in\QQ_p$ written as
\(\alpha=\sum_{j=k}^\infty x_jp^j,  k \in\ZZ \hbox{ and } x_j\in \ZZ\cap \left (-\frac p 2,\frac p 2\right).\)
\end{definition}

Given $\alpha, \beta \in \mathbb Q_p$, we get the corresponding MCF $(\alpha, \beta) = [(a_0, a_1, \ldots), (b_0, b_1, \ldots)]$ by the following iterative equations
\[ \begin{cases} a_k = s(\alpha_k) \cr b_k = s(\beta_k) \cr \alpha_{k+1} = \cfrac{1}{\beta_{k} - b_{k}} \cr \beta_{k+1} = \cfrac{\alpha_k - a_k}{\beta_k - b_k} \end{cases} \]
for $k = 0, 1, \ldots$, with $\alpha_0 = \alpha$ and $\beta_0 = \beta$. If the algorithm does not stop, then the initial values are  represented by the following MCF:
\begin{equation*} \alpha=a_0+\cfrac{b_1+\cfrac{1}{a_2+\cfrac{b_3+\cfrac{1}{\ddots}}{a_3+\cfrac{\ddots}{\ddots}}}}{a_1+\cfrac{b_2+\cfrac{1}{a_3+\cfrac{\ddots}{\ddots}}}{a_2+\cfrac{b_3+\cfrac{1}{\ddots}}{a_3+\cfrac{\ddots}{\ddots}}}} \quad \text{and} \quad \beta=b_0+\cfrac{1}{a_1+\cfrac{b_2+\cfrac{1}{a_3+\cfrac{\ddots}{\ddots}}}{a_2+\cfrac{b_3+\cfrac{1}{\ddots}}{a_3+\cfrac{\ddots}{\ddots}}}}.\end{equation*}
We define the sequences $(A_k)_{k \geq -2}$, $(B_k)_{k \geq -2}$, $(C_k)_{k \geq -2}$ of the numerators and denominators of the convergents, i.e.
\begin{equation} \label{eq:seq-ndc}
[(a_0, \ldots, a_n), (b_0, \ldots, b_n)] = \left( \cfrac{A_n}{C_n}, \cfrac{B_n}{C_n} \right) = (Q_n^{\alpha}, Q_n^\beta)
\end{equation}
as follows
\begin{equation}\label{seq-ABC}
\begin{cases} A_{-2} = 0, \quad A_{-1} = 1, \quad A_0 = a_0 \cr B_{-2} = 1, \quad B_{-1} = 0, \quad B_0 = b_0 \cr C_{-2} = 0, \quad C_{-1} = 0, \quad C_0 = 1 \end{cases} \quad 
\begin{cases} A_n = a_n A_{n-1} + b_n A_{n-2} + A_{n-3} \cr B_n = a_n B_{n-1} + b_n B_{n-2} + B_{n-3} \cr C_n = a_n C_{n-1} + b_n C_{n-2} + C_{n-3} \end{cases}\end{equation}
for any $n \geq 1$. Then 
\begin{equation}\label{matrix}
\prod_{k=0}^n \begin{pmatrix} a_k & 1 & 0 \\ b_k & 0 & 1 \\ 1 & 0 & 0  \end{pmatrix} = \begin{pmatrix} A_n & A_{n-1} & A_{n-2} \\  B_n & B_{n-1} & B_{n-2} \\  C_n & C_{n-1} & C_{n-2} \end{pmatrix}
\end{equation}
for any $n \geq 0$. \\
We define the sequences $(\tilde A_k)_{k \geq -1} = (A_k C_{k-1} - A_{k-1}C_k)$ and $(\tilde B_k)_{k \geq -1} = (B_k C_{k-1} - B_{k-1}B_k)$, arising from the difference between two consecutive convergents:
\[Q_n^\alpha - Q_{n-1}^\alpha = \cfrac{\tilde A_n}{C_n C_{n-1}}, \quad Q_n^\beta - Q_{n-1}^\beta = \cfrac{\tilde B_n}{C_n C_{n-1}}.\]
The following relations hold true:
\begin{equation} \label{eq:seq-num}
\begin{cases} \tilde A_{-1} = 0, \quad \tilde A_{0} = -1, \quad \tilde A_1 = b_1 \cr  \tilde B_{-1} = 0, \quad \tilde B_{0} = 0, \quad \tilde B_1 = 1 \end{cases} \quad 
\begin{cases} \tilde A_n = -b_n \tilde A_{n-1} - a_{n-1} \tilde A_{n-2} + \tilde A_{n-3} \cr \tilde B_n = -b_n \tilde B_{n-1} - a_{n-1} \tilde B_{n-2} + \tilde B_{n-3} \end{cases}
\end{equation}
for any $n \geq 2$. Indeed,
\[\tilde A_n = A_n C_{n-1} - A_{n-1} C_n = (a_n A_{n-1} +b_n A_{n-2} + A_{n-3}) C_{n-1} - A_n(a_n C_{n-1} + b_n C_{n-2} + C_{n-3})\]
\[=-b_n (A_{n-1} C_{n-2} - A_{n-2} C_{n-1}) + A_{n-3} (a_{n-1}C_{n-2}+b_{n-1}C_{n-3}+C_{n-4})-C_{n-3}(a_{n-1}A_{n-2}+b_{n-1}A_{n-3}+A_{n-4})\]
\[= -b_n \tilde A_{n-1} - a_{n-1} \tilde A_{n-2} + \tilde A_{n-3};\]
similarly for the $\tilde B_k$'s. From \eqref{matrix}, we have
\[1 = \det \begin{pmatrix} A_n & A_{n-1} & A_{n-2} \\  B_n & B_{n-1} & B_{n-2} \\  C_n & C_{n-1} & C_{n-2} \end{pmatrix} =\] 
\[= A_n B_{n-1} C_{n-2} - A_{n-1}B_nC_{n-2} - A_nB_{n-2}C_{n-1} + A_{n-2}B_nC_{n-1} + A_{n-1}B_{n-2}C_n - A_{n-2}B_{n-1} C_n,\]
from which
\[\cfrac{1}{C_nC_{n-1}C_{n-2}} = (Q_n^\alpha - Q_{n-1}^\alpha)(Q_n^\beta - Q_{n-2}^\beta) - (Q_n^\beta - Q_{n-1}^\beta)(Q_n^\alpha - Q_{n-2}^\alpha)\]
Moreover, since $\alpha = \cfrac{\alpha_n A_{n-1} + \beta_n A_{n-2} + A_{n-3}}{\alpha_n C_{n-1} + \beta_n C_{n-2} + C_{n-3}}$ and $\beta = \cfrac{\alpha_n B_{n-1} + \beta_n B_{n-2} + B_{n-3}}{\alpha_n C_{n-1} + \beta_n C_{n-2} + C_{n-3}}$, we have
\begin{equation}\label{diff}
(\alpha - Q_{n-1}^\alpha)(\beta - Q_{n-2}^\beta) - (\beta - Q_{n-1}^\beta)(\alpha - Q_{n-2}^\alpha) = \cfrac{1}{C_{n-1}C_{n-2}(\alpha_nC_{n-1}+\beta_nC_{n-2}+C_{n-3})}.\end{equation}

Let us observe that the previous properties hold for general MCFs, while we now give some specific results regarding only $p$--adic MCFs. In the following we will use $\nu_p(\cdot)$ for the $p$--adic valuation, $|\cdot|_p$ for the $p$--adic norm and $|\cdot|_\infty$ for the Euclidean norm. Moreover, we define
\[h_n = \nu_p\left( \cfrac{b_n}{a_n} \right), \quad k_n = \nu_p\left( \cfrac{1}{a_n} \right), \quad K_n = k_1 + \ldots + k_n\]
for any $n \geq 1$ and the sequences $(V_k^{\alpha})_{k \geq -2} = (C_k \alpha - A_k)_{k \geq -2}$, $(V_k^\beta)_{k \geq -2} = (C_k \beta - B_k)_{k \geq -2}$. We recall from \cite{MT} the following properties:
\begin{itemize}
\item $|a_n|_p>1$ and $|b_n|_p<|a_n|_p $, for any $n \geq 1$;
\item $|a_n|_p = |\alpha_n|_p$, $|b_n|_p = \begin{cases} |\beta_n|_p, \ \text{ if } |\beta_n|_p \geq 1 \cr 0, \ \text{ if } \  |\beta_n|_p <1  \end{cases}$,\quad  for any $n \geq 1$;
\item $\nu_p(C_n) = -K_n$, for any $n \geq 1$;
\item $\lim_{n \rightarrow +\infty} |V_n^{\alpha}|_p = \lim_{n \rightarrow +\infty} |V_n^{\beta}|_p = 0$.
\end{itemize}

\section{The quality of the approximations of $p$-adic MCFs} \label{sec:approx}

In this section we investigate how well the convergents of a bidimensional continued fraction approach their limit in $\mathbb Q_p$.

\subsection{The rate of convergence} In first instance we give some results about the rate of convergence of the real sequences $|V_n^\alpha|_p$ and $|V_n^\beta|_p$.

\begin{theorem}
Let $[(a_0, a_1, \ldots), (b_0, b_1, \ldots)]$ be the $p$-adic MCF expansion of $(\alpha, \beta) \in \mathbb Q_p^2$, then 
\[\nu_p(\alpha - Q_n^{\alpha}) \geq K_n + \lfloor \cfrac{n+2}{2} \rfloor, \quad \nu_p(\beta - Q_n^{\beta}) \geq K_n + \lfloor \cfrac{n+3}{2} \rfloor.\]
\end{theorem}

\begin{proof}
We will prove by induction that 
\[\nu_p\left( Q_{n+1}^\alpha - Q_n^\alpha \right) = \nu_p\left( \cfrac{\tilde A_{n+1}}{C_{n+1}C_n} \right) \geq -\nu_p(C_n) + \lfloor \cfrac{n+2}{2} \rfloor,\]
i.e., we have to prove that
\[\nu_p(\tilde A_{n+1}) \geq \nu_p(C_{n+1}) + \lfloor \cfrac{n+2}{2} \rfloor\]
for any $n \geq -2$. We can observe that
\[ \nu_p(\tilde A_{-1}) = \nu_p(C_{-1}) = \infty, \]
for $n=-2$, and
\[ \nu_p(\tilde A_0) = \nu_p(\tilde C_0) = 0, \]
for $n=-1$. Moreover,
\[ \nu_p(\tilde A_1) = \nu_p(b_1), \quad \nu_p(C_1) = \nu_p(a_1), \]
for $n=0$, and we know that $\nu_p(b_1) > \nu_p(a_1)$. Now, we proceed by induction. Consider
\[\nu_p(\tilde A_{n+1}) = \nu_p(-b_{n+1} \tilde A_n - a_n \tilde A_{n-1} + \tilde A_{n-2}) \geq \inf \{ \nu_p(b_{n+1}\tilde A_n), \nu_p(a_n \tilde A_{n-1}), \nu_p(\tilde A_{n-2}) \},\]
by inductive hypothesis we have
\[\nu_p(b_{n+1}\tilde A_n) \geq \nu_p(b_{n+1}) + \nu_p(C_n)+ \lfloor \cfrac{n+1}{2} \rfloor \geq \nu_p(a_{n+1}) + 1 + \nu_p(C_n) + \lfloor \cfrac{n+1}{2} \rfloor \geq \nu_p(C_{n+1}) + \lfloor \cfrac{n+2}{2} \rfloor. \]
Similarly,
\[\nu_p(a_n \tilde A_{n-1}) \geq \nu_p(a_n) + \nu_p(C_{n-1}) + \lfloor \cfrac{n}{2} \rfloor \geq \nu_p(C_{n+1}) + \lfloor \cfrac{n}{2} \rfloor + 1 = \nu_p(C_{n+1}) + \lfloor \cfrac{n+2}{2} \rfloor\]
and
\[\nu_p(\tilde A_{n-2}) \geq \nu_p(C_{n-2}) + \lfloor \cfrac{n-1}{2} \rfloor \geq \nu_p(C_{n+1}) + \lfloor \cfrac{n-1}{2} \rfloor + 3 \geq \nu_p(C_{n+1}) + \lfloor \cfrac{n+2}{2} \rfloor.\]
Thus, we also have $\nu_p(Q_{n+k}^\alpha - Q_n^\alpha) \geq -\nu_p(C_n) + \lfloor \frac{n+2}{2} \rfloor$ and for $k \rightarrow \infty$, we have $\nu_p(\alpha - Q_n^\alpha) \geq -\nu_p(C_n) + \lfloor \frac{n+2}{2} \rfloor$.\\
Similar arguments hold for proving $\nu_p(Q_{n+1}^\beta - Q_n^\beta) \geq -\nu_p(C_n) + \lfloor \frac{n+3}{2} \rfloor$, i.e., for proving $\nu_p(\tilde B_{n+1}) \geq \nu_p(C_{n+1}) +\lfloor \cfrac{n+3}{2} \rfloor$. We just check the basis of the induction:
\[\nu_p(\tilde B_{-1}) = \infty, \quad \nu_p(\tilde B_0) = \infty, \quad \nu_p(\tilde B_1) = 0\]
and
\[\nu_p(C_{-1}) = \infty, \quad \nu_p(C_0) = 0, \quad \nu_p(C_1) = \nu_p(a_1) < 0.\]

\end{proof}

\begin{corollary} \label{cor:1}
	Let $[(a_0, a_1, \ldots), (b_0, b_1, \ldots)]$ be the $p$-adic MCF expansion of $(\alpha, \beta) \in \mathbb Q_p^2$, then 
	\[\nu_p(V_n^\alpha) \geq \lfloor \cfrac{n+2}{2} \rfloor, \quad \nu_p(V_n^\beta) \geq \lfloor \cfrac{n+3}{2} \rfloor, \]
so that
		$$\min\{\nu_p(V_n^\alpha), \nu_p(V_n^\beta)\}\geq \lfloor \cfrac{n+2}{2} \rfloor =\lfloor \cfrac{n}{2} \rfloor +1.$$
\end{corollary}

\begin{remark}
In the real case, given $(\alpha, \beta) = [(a_0, a_1, \ldots), (b_0, b_1, \ldots)]$ it is well--known that
\[\left\vert\alpha - \cfrac{A_n}{C_n}\right\vert_\infty < \cfrac{1}{|C_n|_\infty}, \quad \left\vert\beta - \cfrac{B_n}{C_n}\right\vert_\infty < \cfrac{1}{|C_n|_\infty}.\]
In the $p$--adic case, a stronger result holds, indeed from the previous theorem  we have
\[\left\vert\alpha - \cfrac{A_n}{C_n}\right\vert_p < \cfrac{1}{k|C_n|_p}, \quad \left\vert\beta - \cfrac{B_n}{C_n}\right\vert_p < \cfrac{1}{k|C_n|_p}\]
where $k = k(n)$ tends to infinity.
\end{remark}

On the other hand formula \eqref{diff} implies
	$$\nu_p(V_{n-1}^\alpha V_{n-2}^\beta-V_{n-1}^\beta V_{n-2}^\alpha)=K_n$$
which provides an upper bound for the $p$-adic valuation of the $V_n$'s, namely
\begin{equation}\label{eq:upperbound} \min\{\nu_p(V_n^\alpha), \nu_p(V_n^\beta)\} + \min\{\nu_p(V_{n-1}^\alpha), \nu_p(V_{n-1}^\beta)\}\leq K_{n+1}  
\end{equation}
This shows that the lower bound for $\min\{\nu_p(V_n^\alpha), \nu_p(V_n^\beta)\} $ provided by Corollary \ref{cor:1} is optimal, in the sense that it is reached in some cases:
\begin{example}\label{ex:cattivaappr} Consider an infinite MCF such that $\nu_p(a_n)=-1$ for every $n\geq 1$. Then $K_n=n$ for every $n$, so that by Corollary \ref{cor:1} and formula \eqref{eq:upperbound} we get 
	$$ \min\{\nu_p(V_n^\alpha), \nu_p(V_n^\beta)\} + \min\{\nu_p(V_{n-1}^\alpha), \nu_p(V_{n-1}^\beta)\}=n+1$$ 
so that $\min\{\nu_p(V_n^\alpha), \nu_p(V_n^\beta)\}= \lfloor \cfrac{n}{2} \rfloor +1$ for every $n\ge 1$.
	\end{example}
However, in many other cases the bound provided by Corollary \ref{cor:1} can be improved, as stated by the following propositions:

\begin{proposition}\label{prop:costrsupera} Let $(\ell_n)_{n\geq 0}$ be a sequence of natural numbers $>0$; put $\ell_{-1}=\ell_{-2}=0$ and define $f(n)=\sum_{j=0}^n \ell_n$.
	Let  $[(a_0, a_1, \ldots), (b_0, b_1, \ldots)]$ be an infinite  $p$-adic MCF satisfying $h_{n+1}\geq\ell_{n}, k_{n+1}\geq \ell_n+\ell_{n-1}$ for $n\geq 0$. Then for every $n\in\NN$
	$$\min\{\nu_p(V_n^\alpha), \nu_p(V_n^\beta)\}\geq f(n).$$
	\end{proposition}
\begin{proof}
For $n\geq 1$,  either $\nu_p(\beta_n)>0$, $b_n=0$, $\nu_p\left(\frac {\beta_{n}}{\alpha_{n}}\right )>k_{n}$ or $\nu_p(\beta_n)=\nu_p(b_n)\leq 0$, $\nu_p\left(\frac {\beta_{n}}{\alpha_{n}}\right )=\nu_p\left(\frac {b_{n}}{a_{n}}\right )=h_n$. In any case 
$\nu_p\left(\frac {\beta_{n+1}}{\alpha_{n+1}}\right )\geq \ell_{n}$, for $n\geq 0$.
Let $V_n$ be either $V_n^\alpha$ or $V_n^\beta$. From the formula
\begin{equation}\label{eq:solita1} V_n=-\frac{\beta_{n+1}}{\alpha_{n+1}} V_{n-1}-\frac 1 {\alpha_{n+1}} V_{n-2}, \end{equation}
we get for $n\geq  0$
$$ \frac {V_n} {p^{f(n)}}=\mu_n\frac {V_{n-1}}{p^{f(n-1)}}+\nu_n\frac {V_{n-2}}{p^{f(n-2)}}$$
where
$$
\mu_n =-\frac{\beta_{n+1}}{\alpha_{n+1}}\cdot \frac 1 {p^{\ell_n}},\quad 
\nu_n  = -\frac{1}{\alpha_{n+1}}\cdot \frac 1 {p^{\ell_n+\ell_{n-1}}}
\in\ZZ_p.$$
Since $V_{-1}, V_{-2}\in\ZZ_p$ we obtain by induction $\frac {V_n}{p^{f(n)}}\in\ZZ_p$.
\end{proof}
\begin{corollary}\label{cor:supera}
Let $f:\NN \rightarrow \NN$ be any function. There are infinitely many $(\alpha,\beta)\in\QQ_p^2$ satisfying
$$\min\{\nu_p(V_n^\alpha), \nu_p(V_n^\beta)\}\geq f(n).$$
\end{corollary}
\begin{proof} Of course we can assume $f(n)$ strictly increasing, so that $f(n)=\sum_{j=0}^n \ell_n$ with $\ell_n\in \NN, \ell_n>0$; the proof follows from Proposition \ref{prop:costrsupera} by observing that there are infinitely many
  $p$-adic MCF satisfying $h_{n+1}\geq\ell_{n}, k_{n+1}\geq \ell_n+\ell_{n-1}$ for $n\geq 0$. 
\end{proof}

We would like to investigate in which sense and to which extent the approximations given by $p$-adic convergents may be considered \lq\lq good approximations\rq\rq. Observe that the Browking $s$-function is locally constant, hence so is the function $\QQ_p^2\to \QQ^2$ associating to a pair $(\alpha,\beta)$ its $n$-th convergents $(Q_n^\alpha,Q_n^\beta)$ (where this function is defined). Therefore every $(\alpha,\beta)\in\QQ_p^2$ having a MCF of lenght $\geq n$, has a neighbourhood $U$ such that every $(\alpha',\beta')\in U$ has the same  $k$-convergents than $(\alpha,\beta)$ for $k\leq n$. The following proposition will provide an explicit radius for this neighbourhood.

\begin{proposition}\label{prop:goodapprox}
	Let $(\alpha, \beta)\in \QQ_p^2$ be such that the associated MCF $[(a_0, a_1, \ldots), (b_0, b_1, \ldots)]$ has lenght $\geq n$. Let  $(\alpha', \beta')\in\QQ_p^2$. If $\max\{|\alpha - \alpha'|_p, |\beta - \beta'|_p\} < \cfrac{1}{p^{2K_n}}$, then the MCF   $[(a_0', a_1', \ldots), (b_0', b_1', \ldots)]$ associated to $(\alpha', \beta')$ has lenght $\geq n$ and  $a_i = a_i'$, $b_i = b_i'$, for $i = 0, \ldots, n$.
\end{proposition}
\begin{proof}
	Notice that $ \cfrac{1}{p^{2K_n}}= \cfrac{1}{|C_{n}|^2_p}$.   We prove the thesis by induction on $n$. The claim is certainly true for $n=0$, since in general
	$$|x-y|_p < 1 \Leftrightarrow s(x) = s(y).$$
	Suppose now $n\geq 1$, and $\max\{|\alpha - \alpha'|_p, |\beta - \beta'|_p\} < \cfrac{1}{|C_{n+1}|^2_p}$. By the case $n=0$ we have $a'_0=a_0$, $b'_0=b_0$. Moreover we observe that our hypothesis implies $|\beta-\beta'|<\frac 1 {|a_1|_p}=|\beta-b_0|_p$.  By the properties of the non-archimedean norm, we have
	$$\frac 1{|a'_1|_p} =|\beta'-b_0|_p=\max\{|\beta'-\beta|_p,|\beta-b_0|_p\}=|\beta-b_0|_p=\frac 1{|a_1|_p},$$
	so that $|a_1|_p=|a'_1|_p$. We have
	\begin{equation} \label{eq:diciassette}
	|\alpha_1-\alpha'_1| = \left| \frac 1 {\beta-b_0} -\frac 1 {\beta'-b_0}\right |_p
	=|a_1|^2_p |\beta-\beta'|_p <\prod_{j=2}^{n+1} \frac 1 {|a_j|_p^2} = \cfrac{1}{|C_n^{(1)}|^2_p} 
	\end{equation} 
	where $C_n^{(1)}$ is the $n$--th denominator of the convergents of the MCF expansion of $(\alpha_1, \beta_1)$. Moreover,
	\begin{equation*} 
	|\beta_1-\beta'_1|_p = \left| \alpha_1 (\alpha-a_0) -\alpha'_1(\alpha'-a_0)\right |_p = |(\alpha-a_0)(\alpha_1-\alpha'_1)+\alpha'_1(\alpha-\alpha')|_p \leq \end{equation*}
	\begin{equation} \label{eq:diciotto}
	\leq\max\{|(\alpha-a_0)(\alpha_1-\alpha'_1)|,|a_1|_p|(\alpha-\alpha')|_p \} < \cfrac{1}{|C_n^{(1)}|^2_p}.
	\end{equation}
	Thus, by inductive hypothesis we have $a_i=a'_i, b_i=b'_i$ for $i=1,\ldots, n+1$.
\end{proof}
Unfortunately,  in general the pair $(Q_n^\alpha,Q_n^\beta)$ does not lie in the $p$-adic ball centered in $(\alpha,\beta) $ and having radius $\cfrac 1{p^{2K_n}}$, as Example \ref{ex:cattivaappr} shows. The next proposition gives a constructive sufficient condition ensuring this property.

\begin{proposition} Consider an infinite MCF such that  $k_{n+1}> k_n+k_{n-1}$ and $h_n> k_{n-1}$ for $n\geq 2$. Then for every $n\in\NN$, $\max\{|\alpha - Q_n^\alpha|_p, |\beta - Q_n^\beta|_p\} < \cfrac{1}{p^{2K_n}}$.
			\end{proposition}
	\begin{proof}
	It is a consequence of Proposition \ref{prop:costrsupera} .
	\end{proof}

\subsection{Diophantine study} In this section we want to relate the rate of approximation of the convergents of a $p$-adic MCF to the euclidean size of its numerators and denominators. First, we give a bound on this size.

\begin{lemma}\label{lem:lemmaapprox2} Let $(a_n)_{n\in \NN} $ be a sequence of real numbers, such that there exists $m\in\NN$,  $c_0,\ldots, c_m$ positive real numbers such that $c_m>0$ and
$$|a_{n+m+1}|_\infty  < c_{m}|a_{n+m}|_\infty +c_{m-1}|a_{n+m-1}|_\infty+\ldots + c_{0}|a_n|_\infty. $$
Let $\tilde x$ be the (unique, by the cartesian rule of signs) positive real root of the polynomial
\begin{equation}\label{eq:polinomio} f(X)=X^{m+1}-c_mX^m-\ldots - c_1X-c_0 \end{equation}
and let $M\geq \max\{|a_0|_\infty, \frac{|a_1|_\infty}{\tilde x},\ldots,\frac{|a_m|_\infty}{\tilde x^m}\}.$ Then $ |a_n|_\infty \leq M\tilde x^n$  for every $n\in\NN$.
\end{lemma}
\begin{proof} The proof is straightforward by induction on $n$. \end{proof}

Notice that $f(0)=-c_0< 0$, so that $\tilde x>0$, more precisely $$\tilde x =c_m +\frac {c_{m-1}}{\tilde x} + \frac {c_{m-2}}{\tilde x^2}+\ldots +\frac {c_{0}}{\tilde x^m}$$
which implies $c_m<\tilde x$. Put $C=\sum_{i=0}^m |c_i|_\infty$, if $C<1$, then $f(1)=1-C>0$, so that $0<\tilde x <1$, and we can conclude that $c_m<\tilde x<1.$ In the following, $\tilde x$ will be the real root of the polynomial
$$X^3-\frac 1 2  X^2 -\frac  1 {2p} X -\frac 1 {p^3}$$
so that $\frac 1 2 <\tilde x< 1$ and $\lim_{p \rightarrow \infty} \tilde x = \frac{1}{2}$ in $\mathbb R$. Notice that $p\tilde x$ is the real root of the polynomial
$$X^3-\frac p 2 X^2-\frac p 2 X-1.$$
By specializing to the case of $p$-adic $MCF$ we obtain the following proposition.
\begin{proposition}\label{cor:stimaABC2} 
Given the sequences $(A_n)$, $(B_n)$, $(C_n)$ as in \eqref{eq:seq-ndc}, there exists $H>0$ such that
$$\max\{|A_n|_\infty,|B_n|_\infty, |C_n|_\infty\}   \leq H(p \tilde x)^n,$$
for every $n\in\NN$ and in particular 
$$\max\{|A_n|_\infty,|B_n|_\infty, |C_n|_\infty\}  =o(p^n).$$
\end{proposition}

\begin{proposition}\label{prop:sifermaconpassi}
	Let $\boldsymbol{\alpha}=(\alpha,\beta)\in\QQ^2$, and write
	$$\alpha=\frac{x_0}{z_0}, \quad\quad \beta=\frac{y_0}{z_0}$$
	with $z_0\in\ZZ$ and $x_0,y_0\in\ZZ\left[\frac 1 p\right]$.\\
	The $p$-adic Jacobi-Perron algorithm applied to  $\boldsymbol{\alpha}$ stops in a number of steps bounded by $-\frac{\log(M)}{\log(\tilde x)} $ where
	\begin{align}
	M & =\max\left \{ |z_0|_\infty, \frac 1p |y_0|_\infty+\frac 1 2 |z_0|_\infty, \frac 1 {p^2}|x_0|_\infty+\frac 1 {2p} |y_0|_\infty+\left( \frac 1{2p}+\frac 1 4\right)  |z_0|_\infty\right\}\nonumber \\
	& \leq \max\left \{ |z_0|_\infty, \frac 12 (|y_0|_\infty+|z_0|_\infty), \frac 1 4 (|x_0|_\infty +|y_0|_\infty+|z_0|_\infty)\right \}\label{eq:stima2}\\
	&\leq  |x_0|_\infty+|y_0|_\infty+|z_0|_\infty
	\label{eq:stima2bis}\\
	&\leq 3\max\{ |x_0|_\infty,|y_0|_\infty,|z_0|_\infty
	\}\label{eq:stima3}.
	\end{align}
	
\end{proposition}
\begin{proof} The proof is the same as \cite[Theorem 5]{MT}, but we take into account the number of steps. The $p$-adic JP algorithm produces the sequence of complete quotients $(\boldsymbol{\alpha_n})_{n\geq 0}$, where $$\boldsymbol{\alpha_n}=(\alpha_n,\beta_n)\in\QQ^m, \quad\quad  \alpha_n=\frac{x_n}{z_n}, \quad\quad \beta_n=\frac{y_n}{z_n},$$
	and $x_n,y_n,z_n$ are generated by the following rules:
	$$\left\{ \begin{array}{lll}
	x_n &=& a_nz_n+y_{n+1}\\
	y_n &=& b_nz_n+z_{n+1}\\
	z_n &=& x_{n+1}
	\end{array}
	\right.$$
	with $a_n,b_n\in\mathcal{Y}$, $|y_{n+1}|_p,|z_{n+1}|_p<|z_n|_p$.  Then $\frac {z_n}{p^n}\in\ZZ$ and from the formula
	$$z_{n+1}=z_{n-2}-a_{n-1}z_{n-1}-b_nz_n$$
	we get by Lemma \ref{lem:lemmaapprox2}
	\begin{align*}
	\frac {|z_{n+1}|_\infty } {p^{n+1}} &< \frac 1 2 \frac {|z_{n}|_\infty } {p^{n}} + \frac 1 {2p} \frac {|z_{n-1}|_\infty } {p^{n-1}}+\frac 1 {p^3} \frac {|z_{n-2}|_\infty } {p^{n-2}}\\
	& < M'\tilde x^n \end{align*} where $$M'=\max\left\{\frac{|z_{2}|_\infty } {p^{2}},\frac {|z_{1}|_\infty } {p}, {|z_{0}|_\infty }\right\}.$$
	Then $z_{n+1}=0$ when  $\tilde x^n \leq \frac 1 {M'}$.
	We have
	\begin{align*}
	\frac{|z_{1}|_\infty} {p}&= \frac 1 p |y_0-b_0z_0|_\infty\\
	& < \frac 1 p {|y_0|_\infty} + \frac 1 2  |z_0|_\infty\\
	\frac{|z_{2}|_\infty} {p^2}&= \frac 1 {p^2} |y_1-b_1z_1|_\infty \\
	& = \frac 1 {p^2} |(x_0-a_0z_0)-b_1(y_0-b_0z_0)|_\infty\\
	&< \frac 1 {p^2} {|x_0|_\infty}+\frac 1 {2p}{|y_0|_\infty} +\left (\frac 1 {2p} + \frac 1 4 \right) {|z_0|_\infty}.
	\end{align*}
	Therefore $M'<M$, so that 
	\begin{align}\label{eq:passifermata} z_{n+1}=0 &\hbox{ for } \tilde x^n \leq \frac 1 {M}\\
	&\hbox{ that is for } n \geq  -\frac{\log(M)}{\log(\tilde x)}.\nonumber
	\end{align} 
	Inequalities \eqref{eq:stima2} and \eqref{eq:stima3} are straightforward.
\end{proof}

\begin{corollary}\label{cor:cordiofuno}
	Let $t,u\in\ZZ[\frac 1 p]$ and $v\in\ZZ$ such that the $p$-adic MCF for $(\frac t v, \frac u v)$ has lenght $\geq n+1$.  Then 
	$$\max\{|t|_\infty, |u|_\infty, |v|_\infty\} \geq \frac 1 {3\tilde x^n}.$$
\end{corollary}
\begin{proof}
	With the notation of the proof of Proposition \ref{prop:sifermaconpassi}, we have $z_{n+1}\not=0$, then the claim follows from by \eqref{eq:stima3} and \eqref{eq:passifermata}.
\end{proof}
\begin{corollary}\label{cor:cordiofdue} Let $(\alpha,\beta)\in\QQ_p^2$ be a pair having a $p$-adic MCF expansion of lenght $\geq  n+1$. Then
	$$\max\{|A_n|_\infty, |B_n|_\infty, |C_n|_\infty\} \geq \frac 1 {3 p^{K_n} \tilde x^n}.$$
\end{corollary}
\begin{proof}
	If we set $t=p^{K_n}A_n, u=p^{K_n}B_n, v=p^{K_n}C_n$, then the hypothesis of Corollary \ref{cor:cordiofuno} is fulfilled. 
\end{proof}
The following theorem establishes an explicit lower bound for the euclidean lenght of a pair of rational numbers which is a \lq\lq good approximation\rq\rq of a $p$-adic pair w.r.t the corresponding $K_n$.
\begin{theorem}\label{cor:cordioftre}
Let $(\alpha, \beta)\in \QQ_p^2$ be a pair having a $p$-adic MCF expansion of  lenght $\geq n+1$. Let  $(\frac tv, \frac u v)\in\QQ^2$ with $t,u\in\ZZ[\frac 1 p]$, $v\in \ZZ$ and assume $\max\left\{\left|\alpha - \frac tv\right|_p,\left |\beta - \frac uv\right|_p\right \} < \cfrac{1}{p^{2K_{n+1}}}$; then $\max\{|t|_\infty, |u|_\infty, |v|_\infty\}\geq  \frac 1 {3\tilde x^n}$.
\end{theorem}

\begin{proof} By Proposition \ref{prop:goodapprox} the pair $(\frac tv, \frac u v)$ has the same $MCF$ expansion as $(\alpha,\beta)$ up to $n+1$. The claim follows from Corollary \ref{cor:cordiofuno}.
\end{proof}

\section{Results related to algebraic dependence} \label{sec:alg}
\subsection{A $p$-adic Liouville type theorem on algebraic dependence}
The quality of rational approximations to real numbers is related to their algebraic dependence. Indeed, if it is possible to find infinitely many good approximations to a $m$--tuple of real numbers, then they are algebraically independent, see, e.g., \cite{Adams2}. Similar results also hold for the $p$--adic numbers \cite{Lao}. In the following theorem, we prove a new result of this kind and then we apply it to $p$--adic MCFs.

\begin{lemma}\label{lem:silly}
Let $C$ be a non-zero integer number, then
$$ |C|_p\geq \frac 1 {|C|_\infty}.$$
\end{lemma}
\begin{proof}
The result follows from $p^{\nu_p(C)}\leq |C|_\infty$ and $|C|_p= \frac 1 {p^{\nu_p(C)}}$.
\end{proof}
The following result is a variant of \cite[Theorem 3] {Lao}.
\begin{theorem}\label{teo:dipalg}
Given $\alpha,\beta\in\QQ_p\setminus\QQ$ such that $F(\alpha,\beta)=0$, for $F(X,Y)\in\mathbb{Z}[X,Y]$ non-zero polynomial with minimal total degree $D$, let $(t_n)_{n \geq 0}, (u_n)_{n \geq 0}, (v_n)_{n \geq 0}$ be sequences of integers such that $v_n\not=0$ for every $n\in\NN$ and
\begin{equation}\label{eq:convergono2} \lim_{n \rightarrow \infty}\frac{u_n}{v_n} = \alpha,\quad \lim_{n \rightarrow \infty} \frac{t_n}{v_n}=\beta \end{equation} 
in $\mathbb Q_p$. Consider $M_n =\max\{|t_n|_\infty,|u_n|_\infty,|v_n|_\infty\}$ and $U_n = \max\left\{\left | \alpha- \frac {t_n}{v_n}\right |_p, \left | \beta- \frac {u_n}{v_n}\right |_p \right\}$; if
\begin{equation}\label{eq:hypoUMD} \lim_{n \rightarrow \infty} U_n\cdot M_n^D  = 0 \end{equation}
in $\mathbb R$, then $F\left (\frac {t_n}{v_n},\frac {u_n}{v_n}\right )=0$ for $n\gg 0$.
\end{theorem}

\begin{proof}
We observe that $ v_n^D \cdot F\left ( \frac {t_n}{v_n}, \frac {u_n}{v_n}\right )\in\ZZ$ and
\begin{align}
    \left | v_n^D \cdot F\left ( \frac {t_n}{v_n}, \frac {u_n}{v_n}\right )\right |_\infty \leq KM_n^D,\label{eq:kappaIII}
\end{align}
where $K$ is the sum of the Euclidean absolute values of the coefficients of $F(X,Y)$.
Therefore, if $F\left ( \frac {t_n}{v_n}, \frac {u_n}{v_n}\right )$ is not zero, we have

\begin{equation} \label{eq:kappa4I}
    \left |F\left ( \frac {t_n}{v_n}, \frac {u_n}{v_n}\right )\right |_p \geq \left | v_n^D\cdot F\left ( \frac {t_n}{v_n}, \frac {u_n}{v_n}\right )\right |_p\geq \frac 1 {\left | v_n^D \cdot F\left ( \frac {t_n}{v_n}, \frac {u_n}{v_n}\right )\right |_\infty }\geq \frac 1 {KM_n^D}\end{equation}

by Lemma \ref{lem:silly} and \eqref{eq:kappaIII}. On the other hand, we can write
    $$F(X,Y)=\sum_{i,j}A_{ij}(X-\alpha)^i(Y-\beta)^j$$
    with $A_{ij}\in\QQ_p$. We have 
    $$A_{10}=\frac{ \partial F}{\partial X}(\alpha, \beta), \quad A_{01}=\frac{ \partial F}{\partial Y}(\alpha, \beta), $$ 
    if $A_{10}=0$ and $\frac{ \partial F}{\partial X}(X, Y)\not=0$ then the latter polynomial  would give an algebraic dependence relation between $\alpha$ and $\beta$ of total degree $\leq D-1$, therefore $\frac{ \partial F}{\partial X}(X, Y)=0$, that is $X$ does not appear in $F(X,Y)$. Analogously $A_{01}=0$ implies that $\frac{ \partial F}{\partial Y}(X, Y)=0$, that is $Y$ does not appear in $F(X,Y)$. It follows that if $A_{ij}\not=0$  for some $i >0  0$ then $A_{10}\not=0$; and if $A_{ij}\not=0$  for some $j >0$ then $A_{01}\not=0$. Hence, it is easy to see that for every $i,j$ such that $i+j>1$ and $n\gg 0$
    $$\nu_p\left (A_{ij}\left (\frac{t_n}{v_n}-\alpha\right )^i\left (\frac{u_n}{v_n}-\beta\right )^j\right ) > \min\left \{ \nu_p\left (A_{10}\left (\frac{t_n}{v_n}-\alpha\right )\right ) , \nu_p\left (A_{01}\left (\frac{u_n}{v_n}-\beta\right )\right) \right \}.$$
    Therefore for $n\gg 0$, we obtain

\begin{align} 
    \left |F\left ( \frac {t_n}{v_n}, \frac {u_n}{v_n}\right )\right |_p & \leq\max_{ij}\left\{\left  | A_{ij}\left (\frac{t_n}{v_n}-\alpha\right )^i\left (\frac{u_n}{v_n}-\beta\right )^j\right |\right\}\nonumber \\&= 
    \max\left\{ \left |A_{01}\left (\frac{t_n}{v_n}-\alpha \right )\right  |_p,\left |A_{10}\left (\frac{u_n}{v_n}-\beta\right )\right |_p\right \} \leq H \cdot U_n, \label{eq:acca4I}
    \end{align}
for $H=\max\{|A_{01}|_p, |A_{10}|_p\}$. Putting together equations \eqref{eq:kappa4I} and \eqref{eq:acca4I}, we get
    $$\frac 1 {KM_n^D} \leq  \left |F\left ( \frac {t_n}{v_n}, \frac {u_n}{v_n}\right )\right |_p \leq H\cdot U_n,$$ 
for every $n$ such that $F\left ( \frac {t_n}{v_n}, \frac {u_n}{v_n}\right )\not=0$. This implies that there exists $C>0$ such that if $F\left ( \frac {t_n}{v_n}, \frac {u_n}{v_n}\right )\not=0$ then $U_n\cdot M_n>C$. Then hypothesis \eqref{eq:hypoUMD} proves the claim.
\end{proof}

\begin{remark}\label{rem:tuv}
We shall apply Theorem \ref{teo:dipalg} with
\begin{equation*}\label{eq:tuv}
\frac{u_n}{v_n} = Q_n^{\alpha}, \quad \frac{t_n}{v_n} = Q_n^{\beta},
\end{equation*}
with  $u_n,t_n,v_n\in\ZZ$ coprime.
Set $\delta=\max\{0,-v(\alpha),-v(\beta)\},$
then for $n\gg 0$, we have
\begin{equation*}\label{eq:valtuv}
(t_n,u_n, v_n)=p^{K_n+\delta}(A_n,B_n,C_n).
\end{equation*}
Consequently, if  $M_n=\max\{|t_n|_\infty,|u_n|_\infty,|v_n|_\infty\}$ then by Corollary \ref{cor:stimaABC2} there exists $H>0$ such that

\begin{equation}\label{eq:stimaM_n}
M_n\leq Hp^{K_n}\tilde z^n =o( p^{n+K_n}).
\end{equation}
\end{remark}

\subsection{Some consequences on linear dependence}

We specialize Theorem \ref{teo:dipalg} to the case $D=1$, i.e., when we have linear dependence. In \cite{MT}, the authors proved that if the $p$--adic Jacobi--Perron algorithm stops in a finite number of steps, then the initial values are $\mathbb Q$--linearly dependent. Further results about linear dependence and $p$--adic MCFs can be found in \cite{MT2}, where it is conjectured that if we start the $p$--adic Jacobi--Perron algorithm with a $m$--tuple of $\mathbb Q$--linearly dependent numbers, then the algorithm is finite or periodic. Here, exploiting the previous results, we can give a condition that ensures the finiteness of the $p$--adic Jacobi--Perron algorithm when it processes certain $\mathbb Q$--linearly dependent inputs.

\begin{theorem}\label{teo:diplin}
Given $\alpha, \beta \in \mathbb Q_p$, consider
\[ M_n = \max \{ |A_n|_\infty, |B_n|_\infty, |C_n|_\infty \}, \quad U_n = \max \left\{\left| \alpha - \cfrac{A_n}{C_n} \right|_p, \left| \beta - \cfrac{B_n}{C_n} \right|_p\right\},\]
where $(A_n)$, $(B_n)$, $(C_n)$ are the sequences of numerators and denominators of convergents of the MCF representing $(\alpha, \beta)$. If
\[ \lim_{n \rightarrow \infty} U_n\cdot M_n = 0, \]
then either $\alpha,\beta,1$ are linearly independent over $\QQ$ or the $p$-adic MCF expansion of $(\alpha,\beta)$ is finite.
\end{theorem}

\begin{proof}
Assume that the $p$-adic MCF for $(\alpha,\beta)$ is not finite, then the sequence  $(Q_n^\alpha, Q_n^\beta)$ $p$-adically converges to $(\alpha,\beta)$  by \cite[Proposition 3]{MT} and $(\alpha,\beta)\not\in\QQ^2$, by \cite[Theorem 5]{MT}. 
Suppose that $A\alpha+B\beta+C=0$ for some $A,B,C\in\QQ$ not all zero. We define the sequence $S_n=AA_{n-1}+BB_{n-1}+CC_{n-1}$; Theorem \ref{teo:dipalg} implies that $S_n=0$ for $n$ sufficiently large. Furthermore, it is straightforward to see that $ S_n=AV_{n-1}^\alpha+B V_{n-1}^\beta$ (see also \cite{MT2}) and by Corollary \ref{cor:1} we should have $V_n^\alpha=V_n^\beta=0$, which is a contradiction, as $(\alpha,\beta)\not\in\QQ^2$.
\end{proof}

\begin{remark}
     Theorem \ref{teo:diplin} is an improvement of a result implicitly contained in \cite[Proposition 10]{MT2}, namely that
     if $(1,\alpha,\beta)$ are linearly dependent over $\QQ$ and there is a constant $K>0$ such that  \begin{equation} \label{eq:hypoprec} \max\{|V_n^\alpha|_p,|V_n^\beta|_p\}\leq \frac K {p^n}\end{equation} then the $p$-adic Jacobi-Perron algorithm stops in finitely many steps when applied to $(\alpha,\beta)$.\\
     In fact \eqref{eq:hypoprec} implies
    $$U_n\cdot p^{K_n}\leq \frac K {p^n},$$
    so that, by \eqref{eq:stimaM_n},
    $$U_n\cdot M_n\leq H\tilde z^n U_np^{K_n}\leq 
    KH \left( \frac {\tilde z} p\right)^n\buildrel\infty\over\rightarrow 0\hbox{ for } n\to\infty.$$
\end{remark}

\subsection{A class of fast convergent $p$-adic MCFs}
Finally, we see some conditions on the partial quotients that produce MCFs converging to algebraically independent numbers or having convergents that satisfy an algebraic relation.

\begin{lemma}\label{lem:costralgindD}  
Given a MCF $[(a_0, a_1, \ldots), (b_0, b_1, \ldots)]$ such that
\begin{align}
    &k_{n+1} \geq (D-1)(k_n+k_{n-1})+2D;\label{eq:kgrandeD}\hbox{ and } \\
    &h_{n+1}\geq  (D-1)k_n+D.\label{eq:condizsubD}
\end{align}
for $n \gg 0$ and $D \geq 1$, then there exists $C\in\mathbb{R}$ such that
$$\min\{v_p(V_n^\alpha),v_p(V_n^\beta)\}\geq (D-1)K_n+Dn +C,\quad\hbox{ for } n\in \mathbb{N}.$$
\end{lemma}
\begin{proof}
The argument is the same as in the proof of Proposition \ref{prop:costrsupera}. In fact, if conditions \eqref{eq:kgrandeD} and \eqref{eq:condizsubD}   hold for every $n\geq 0$t hen the claim direcly follows from Proposition \ref{prop:costrsupera} by putting $\ell_n=(D-1)k_n+D$.
In any case hypotheses \eqref{eq:kgrandeD} and \eqref{eq:condizsubD} imply that
\begin{equation} \label{eq:betasualfaD} v_p\left(\frac {\beta_{n+1}}{\alpha_{n+1}}\right )\geq (D-1)k_n+D, \quad\hbox{ for } n\gg0.\end{equation}
Let $V_n$ be one of $V_n^\alpha$, $V_n^\beta$. From the formula \eqref{eq:solita1}
we get for $n\gg 0$
$$ \frac {V_n} {p^{(D-1)K_n+Dn}}= \mu_n\frac {V_{n-1}}{p^{(D-1)K_{n-1}+D(n-1)}}+\nu_n\frac {V_{n-2}}{p^{(D-2)K_{n-1}+D(n-2)}}$$
$$
    \mu_n = -\frac{\beta_{n+1}}{\alpha_{n+1}}\frac 1 {p^{(D-1)k_n+D}},\quad\qquad 
    \nu_n  = -\frac 1{\alpha_{n+1}}\frac 1 {p^{(D-1)(k_n+k_{n-1})+2D}}.
$$
By \eqref{eq:kgrandeD} and \eqref{eq:betasualfaD}  there exists $n_0$ such that $\mu_n,\nu_n\in\ZZ_p$ for $n > n_0$.  Then   
$$\left |\frac {V_n} {p^{(D-1)K_n+Dn}}\right |_p\leq \max\left \{ \left |\frac {V_i} {p^{(D-1)K_i+Di}}\right |_p, i=0,\ldots, n_0\right\}\hbox{ for } n>n_0 $$
so that the claim follows by setting
$$C=\min\left\{v_p\left(\frac {V_i^\alpha} {p^{(D-1)K_i+Di}}\right ),v_p\left(\frac {V_i^\beta} {p^{(D-1)K_i+Di}}\right ),i=0 ,\ldots, n_0\right \} .$$
\end{proof}

\begin{theorem}\label{teo:costralgindD} Assume that $\alpha,\beta$ are algebraically dependent and let $F(X,Y)\in\mathbb{Q}[X,Y]$ be a non zero polynomial of minimum total degree $D$ such that $F(\alpha,\beta)=0$. If the MCF expansion of $(\alpha, \beta)$ satisfies conditions
\eqref{eq:kgrandeD} and \eqref{eq:condizsubD}, then $F(Q_n^\alpha,Q_n^\beta)=0$ for $n\gg 0$.
\end{theorem}
\begin{proof}
Let $M_n,U_n$ be as in Theorem \ref{teo:dipalg}. For $n\gg 0$ and a suitable constant $C>0$ we have
\begin{align*} U_n  & =\frac 1 {p^{K_n}} \max\left\{ | V_n^\alpha|_p, |V_n^\beta |_p \right\}\\
&\leq \frac C {p^{D(K_n+n)}}\hbox{ by Lemma \ref{lem:costralgindD}}\\
&\leq \frac C {p^{(K_n+n)D}}\end{align*}
so that $\lim_{n\to\infty} U_n\cdot M_n=0$ by  \eqref{eq:stimaM_n}. 
Then the claim follows from Theorem \ref{teo:dipalg}.
\end{proof}

\begin{theorem}\label{teo:costralgindnew}
Given $(\alpha, \beta) = [(a_0, a_1, \ldots), (b_0, b_1, \ldots)]$ such that
\begin{equation*} \label{eq:kappa}
\lim_{n \rightarrow \infty} \frac {k_n}{k_{n-1}+k_{n-2}} = \infty, \quad \lim_{n \rightarrow \infty} \frac {h_n}{k_{n-1}} = \infty
\end{equation*}
in $\mathbb R$, then either $\alpha,\beta$ are algebraically independent or there exists a non zero polynomial $F(X,Y)\in\mathbb{Q}[X,Y]$ such that $F(Q_n^\alpha,Q_n^\beta)=0$ for $n\gg 0$.
\end{theorem}

\begin{proof}
Let $D>0$. For $n\gg_D 0$ 
we have 
\begin{align*}
 k_n & \geq D(k_{n-1}+k_{n-2})\geq (D-1) D(k_{n-1}+k_{n-2})+2D,\hbox{ and }\\
 h_n &\geq D k_{n-1}\geq (D-1)k_{n-1} +D
\end{align*}
 Then the claim follows from Theorem \ref{teo:costralgindD}.
\end{proof}

\begin{remark}
	By Faltings theorem, an algebraic curve having infinitely many rational points must have genus $0$ or $1$. This is a strong condition on polynomials $F(X,Y)\in\mathbb{Q}[X,Y]$ such that $F(Q_n^\alpha,Q_n^\beta)=0$, for $n\gg 0$.
\end{remark}

\end{document}